\documentclass[graybox,envcountsect,envcountsame]{svmult}

\usepackage{mathptmx}       
\usepackage{helvet}         
\usepackage{courier}        
\usepackage{type1cm}        
\usepackage{makeidx}         
\usepackage{graphicx}        
\usepackage{multicol}        
\usepackage[bottom]{footmisc}
\usepackage{amsmath,amssymb}
\usepackage{hyperref}

\begin{document}

\title*{On Wavelet-Galerkin Methods for Semilinear Parabolic Equations with Additive Noise}
\author{Mih\'aly Kov\'acs, Stig Larsson and Karsten Urban}
\institute{Mih\'aly Kov\'acs \at University of Otago,
Dept. of Mathematics and Statistics,
P.O. Box 56, Dunedin, New Zealand\\ \email{mkovacs@maths.otago.ac.nz}
\and Stig Larsson \at Chalmers University of Technology,
Mathematical Sciences,
SE--412 96 Gothenburg, Sweden\\ \email{stig@chalmers.se}
\and Karsten Urban \at Ulm University,
Inst.\ for Numerical Mathematics,
Helmholtzstr.\ 18,
DE--89069 Ulm, Germany\\ \email{karsten.urban@uni-ulm.de}}
%
%
\maketitle


\abstract{We consider the semilinear stochastic heat equation perturbed by
additive noise.  After time-discretization by Euler's method the
equation is split into a linear stochastic equation and a non-linear
random evolution equation.  The linear stochastic equation is
discretized in space by a non-adaptive wavelet-Galerkin method.
This equation is solved first and its solution is substituted into
the nonlinear random evolution equation, which is solved by an
adaptive wavelet method.  We provide mean square estimates for the
overall error.}

\newcommand{\bbE}{\mathbb{E}}
\newcommand{\N}{\mathbb{N}}
\newcommand{\R}{\mathbb{R}}
\newcommand{\J}{\mathcal{E}}
\newcommand{\bd}{\mathbf{d}}
\newcommand{\boldf}{\mathbf{f}}
\newcommand{\bg}{\mathbf{g}}
\newcommand{\bu}{\mathbf{u}}
\newcommand{\bv}{\mathbf{v}}
\newcommand{\bw}{\mathbf{w}}
\newcommand{\bA}{\mathbf{A}}
\newcommand{\bB}{\mathbf{B}}
\newcommand{\bD}{\mathbf{D}}
\newcommand{\bE}{\mathbf{E}}
\newcommand{\bI}{\mathbf{I}}
\newcommand{\bK}{\mathbf{K}}
\newcommand{\bM}{\mathbf{M}}
\newcommand{\bR}{\mathbf{R}}
\newcommand{\bQ}{\mathbf{Q}}
\newcommand{\cA}{\mathcal{A}}
\newcommand{\cC}{\mathcal{C}}
\newcommand{\cD}{\mathcal{D}}
\newcommand{\cF}{\mathcal{F}}
\newcommand{\cI}{\mathcal{I}}
\newcommand{\cJ}{\mathcal{J}}
\newcommand{\cP}{\mathcal{P}}
\newcommand{\cR}{\mathcal{R}}
\newcommand{\cT}{\mathcal{T}}

\newcommand{\HS}{\mathrm{HS}}
\newcommand{\dd}{{\mathrm{d}}}
\newcommand{\ee}{{\mathrm{e}}}
\newcommand{\W}{W_A}
\newcommand{\ens}[1]{\bbE \|#1\|^2}
\newcommand{\Ens}[1]{\bbE \Big\|#1\big\|^2}
\newcommand{\shs}[1]{\hs{#1}^2}
\newcommand{\hs}[1]{\|#1\|_{\HS}}

\newcommand{\norm}[2][]{\|{#2}\|_{{#1}}}
\newcommand{\snorm}[2][]{|{#2}|_{{#1}}}
\newcommand{\Snorm}[2][]{\Big|{#2}\Big|_{{#1}}}
\newcommand{\Norm}[2][]{\Big\|{#2}\Big\|_{{#1}}}
\newcommand{\abs}[1]{|#1|}
\newcommand{\inner}[2][]{\langle{#2}\rangle_{{#1}}}
\newcommand{\Inner}[2][]{\Big\langle{#2}\Big\rangle_{{#1}}}
\newcommand{\Supp}{\operatorname{supp}}

\section{Introduction}\label{sec:1}

We consider the following semilinear parabolic problem with additive
noise,
\begin{equation}\label{eq:1}
  \begin{aligned}
&\dd u -\nabla\cdot(\kappa \nabla u)\, \dd t
  = f(u)\,\dd t +\dd W,\quad && x\in\cD,\ t\in (0,T),\\
&u=0,&&x\in\partial\cD,\  t\in (0,T),\\
&u(\cdot,0)=u_0, && x\in\cD .
  \end{aligned}
\end{equation}
Here $T>0$, $\cD\subset\R^d$, $d=1,2,3$, is a convex polygonal domain or a domain with
smooth boundary $\partial\cD$, and $\{W(t)\}_{t\ge 0}$ is an
$L_2(\cD)$-valued $Q$-Wiener process on a filtered probability space
$(\Omega, \cF, \mathbb{P}, \{\cF_t\}_{t\ge 0})$ with respect to the
normal filtration $ \{\cF_t\}_{t\ge 0}$.  We use the notation
$H=L_2(\cD)$, $V=H^1_0(\cD)$ with $\|\cdot\|=\|\cdot\|_H$ and
$(\cdot,\cdot)=(\cdot,\cdot)_H$. Moreover, $A\colon V\to V'$ denotes
the linear elliptic operator $Au=-\nabla\cdot(\kappa\nabla u)$ for
$u\in V$ where $\kappa(x)>\kappa_0>0$ is smooth. As usual we consider the bilinear form $a\colon V\times
V\to\mathbb{R}$ defined by $a(u,v) = \langle Au,v\rangle$ for $u,v\in
V$, and $\langle\cdot,\cdot\rangle$ denotes the duality pairing of
$V'$ and $V$.  We denote by $\ee^{-tA}$ the analytic semigroup in $H$
generated by the realization of $-A$ in $H$ with $D(A)=H^2(\mathcal D)\cap H^1_0(\mathcal{D})$.  Finally, $f\colon H\to H$ is a nonlinear function,
which is assumed to be globally Lipschitz continuous, i.e., there
exists a constant $L_f$ such that
\begin{equation}\label{Lip}
	\|f(u)-f(v)\|\le L_f \|u-v\|,\quad u,v\in H.
\end{equation}

It is well known that our assumptions on $A$ and on the spatial domain
$\mathcal{D}$ implies the existence of a sequence of nondecreasing
positive real numbers $\{\lambda_k\}_{k\geq 1}$ and an orthonormal
basis $\{e_k\}_{k\geq 1}$ of $H$ such that
\begin{equation*}\label{eq:spectral}
Ae_k = \lambda_k e_k, \quad \lim_{k\rightarrow +\infty} \lambda_k = +\infty.
\end{equation*}
Using the spectral functional calculus for $A$ we introduce the fractional powers $A^s$, $s
\in \mathbb{R}$, of $A$ as
\begin{equation*}\label{eq:fp}
A^s v=\sum_{k=1}^{\infty}\lambda_k^s(v,e_k)e_k,\quad
D(A^s)=\Big\{v\in H:\|A^sv\|^2=\sum_{k=1}^{\infty}\lambda_k^{2s}(v,e_k)^2<\infty\Big\}
\end{equation*}
and spaces $\dot{H}^\beta=D(A^{\beta/2})$ with norms
$\norm[\beta]{v}=\norm{A^{\beta/2} v}$. It is classical that if $0\le
\beta < 1/2$, then $\dot{H}^\beta=H^\beta$ and if $1/2<\beta\le 2$,
then $\dot{H}^\beta=\{u\in H^\beta:u|_{\partial \mathcal{D}}=0\}$,
where $H^\beta$ denotes the standard Sobolev space of order $\beta$.
We also use the spaces $L_{2}( \Omega, \dot{H}^{\beta})$ with the mean
square norms $\norm[L_{2}( \Omega,
\dot{H}^{\beta})]{v}=\big(\bbE[\norm[\beta]{v}^2]\big)^{\frac12}$.

We assume for some $\beta\ge0$ that
\begin{equation}  \label{eq:HSassumption}
\norm[\HS]{A^{\frac{\beta-1}{2}}Q^{\frac12}}<\infty, \quad
u_0\in L_2(\Omega, \dot{H}^\beta).
\end{equation}
Here $Q$ is the covariance operator of $W$ and $\norm[\HS]{\cdot}$
denotes the Hilbert-Schmidt norm. The Hilbert-Schmidt condition in
\eqref{eq:HSassumption} can be viewed as a regularity assumption on
the covariance operator $Q$. In particular, it holds with $\beta=1$ if
$Q$ is a trace class operator and with $\beta<1/2$ if $Q=I$ and
$d=1$. More generally, it holds if
$\sum_{k=1}^{\infty}\lambda_k^{-\alpha}<\infty$ (thus $\alpha>d/2$)
and $A^{\beta+\alpha-1}Q$ is a bounded linear operator on $H$ (see,
for example, \cite[Theorem 2.1]{KLLweak}).

It is known (\cite{DapZab}, \cite[Lemma 3.1]{KLL}) that if \eqref{Lip}
and \eqref{eq:HSassumption} hold, then \eqref{eq:1} has a unique mild
solution, which is defined to be the solution of the fixed point
equation
\begin{equation}  \label{eq:01}
  \begin{aligned}
u(t)
&= \ee^{-tA}u_0
+\int_0^t\ee^{-(t-s)A}f(u(s))\,\dd s
+\int_0^t\ee^{-(t-s)A}\,\dd W(s).
\end{aligned}
\end{equation}
This naturally splits the solution as $u=v + w$, where $w$ is
a stochastic convolution,
\begin{equation} \label{eq:03}
w(t)= \int_0^t\ee^{-(t-s)A}\,\dd W(s),
\end{equation}
which is the solution of
\begin{equation} \label{eq:02}
\dd w +Aw\,\dd t=\dd W, \ 0< t\le T;\quad
 w(0)=0,
\end{equation}
and $v$ is the solution of the random evolution equation
\begin{equation}  \label{eq:04}
\dot{v} +Av=f(v+w) , \ 0< t\le T;\quad v(0)=u_{0}.
\end{equation}

Our approach will be to first compute $w$ and then to insert it into
\eqref{eq:04} which we then solve for $v$.  Finally, $u=v+w$.  For the
numerical solution we use Rothe's method, where we first discretize
with respect to time and then discretize the resulting elliptic
problems with wavelet methods.

Thus, we fix a time step $\tau>0$, set $t_n:= n\tau$ with $t_N=T$, and
consider a backward Euler discretization of \eqref{eq:1}.  With
$u^n\approx u(t_n)$ and increments $\Delta W^n=W(t_n)-W(t_{n-1})$ this
reads
\begin{equation}
\label{Eq:2}
u^n + \tau Au^n = u^{n-1} + \tau f(u^n) + \Delta W^n,\ 1\le n\le N;
\quad u^0=u_0.
\end{equation}
Then we decompose $u^n= v^n + w^n$ to get time-discrete
versions of \eqref{eq:02} and \eqref{eq:04}:
\begin{subequations}
\label{Eq:3}
  \begin{align}
 w^n + \tau A w^n&=w^{n-1}+\Delta W^n,\ &1\le n\le N;\quad w^0=0,& \label{Eq:3a} \\
   v^n + \tau Av^n&=v^{n-1}+\tau f(v^n + w^n),\ &1\le n\le N;\quad v^0=u_0.&\label{Eq:3b}
 \end{align}
\end{subequations}
This allows us to solve the linear problem \eqref{Eq:3a} first and
use the result as an input for the nonlinear problem \eqref{Eq:3b}.
Moreover, the stochastic influence in \eqref{Eq:3b} is smoother than
in \eqref{Eq:3a}, which allows us to use fast nonlinear solvers.

We consider now the spatial discretization of \eqref{Eq:3}.  To this
end, let $S_J$ be a multiresolution space of order $m$ (see \eqref{eq:m} for the definition) and let
$\{w_J^n\}_{n=0}^{N}\subset S_J$ be the corresponding Galerkin
approximation of $\{w^n\}_{n=0}^{N}$, i.e.,
\begin{equation}
  \label{eq:multires}
   w_J^n + \tau A_J w_J^n=w_J^{n-1}+P_J\Delta W^n,\ 1\le n\le N;\quad
   w_J^0=0.
\end{equation}
We refer to Sect.\ \ref{sec:multires} for further details.  We enter this
approximation instead of $w^n$ into \eqref{Eq:3b}.  The corresponding
equation reads
\begin{align}
\label{Eq:5}
\bar{v}^n + \tau A \bar{v}^n
=  \bar{v}^{n-1}+\tau f( \bar{v}^n  + w_J^n),
\ 1\le n\le N; \quad \bar{v}^0=u_0.
\end{align}
For each $\omega\in\Omega$ and for each $n\ge1$ the nonlinear equation
in \eqref{Eq:5} is solved by an adaptive wavelet algorithm to yield an
approximate solution $v_\varepsilon^n$ with tolerance
$\varepsilon_n$. Note that we use the same tolerance for each
$\omega$.  More precisely, denoting $\bar{v}^n=E_n(\bar{v}^{n-1})$,
where $E_n=(I+\tau A-\tau f(\cdot+w_J^n))^{-1}$ is the nonlinear
one-step operator from \eqref{Eq:5}, we assume that
$v_\varepsilon^n=\tilde{E}_n(v_\varepsilon^{n-1})$, where $\tilde{E}_n$
is an approximation of $E_n$ such that
\begin{align}
\label{Eq:5b}
\| E_n(v)-\tilde{E}_n(v) \| \leq \varepsilon_n,
\quad 1\le n\le N, \ v\in H.
\end{align}
The output of the computation will then be the sequence
\begin{equation} \label{output}
  u_{\varepsilon}^n=v_{\varepsilon}^n+w_J^n, \quad 0\le n\le N.
\end{equation}
The total error is $u_\varepsilon^n-u(t_n) =
(v_\varepsilon^n-\bar{v}^n) +(\bar{v}^n-v^n) +(w_J^n-w^n)
+(u^n-u(t_n)) $.  The contributions are bounded as follows, where the
constants $C$ depend on $\| u_{0} \|_{L_{2}( \Omega,
  \dot{H}^{\beta})}$, $\| A ^{\frac{\beta -1}{2}}Q^{\frac12}\|_{\HS}$,
and $T$, referring to assumption \eqref{eq:HSassumption}.  We also
assume $\tau L_f<\frac12$.

First, in Sect. \ref{Sec:2.2b}, an adaptive wavelet algorithm is
described which realizes \eqref{Eq:5b}. In Theorem
\ref{thm:complexity}, we also analyze the computational effort of the
algorithm applied to \eqref{Eq:5}.  We conclude the section by showing
that
\begin{align}
\label{Eq:5c}
\max_{0\le t_n\le T}
\norm[L_2(\Omega, H)]{v_\varepsilon^n-\bar{v}^n}
\le C\sum_{n=1}^{N} \varepsilon_n.
\end{align}

The multiresolution approximation of the time-discrete stochastic
convolution is studied in Sect.\ \ref{sec:multires} and Theorem
\ref{Thm:wn} shows that
\begin{equation}  \label{eq:08}
 \max_{0\le t_n\le T}
\norm[L_2(\Omega, H)]{w_J^n-w^n}
 \le C\, 2^{-J\min(\beta,m)}.
\end{equation}

In Sect.\ \ref{sec:time}, Theorem \ref{Thm:tme}, we study the
time-discretization error and prove that
\begin{equation}  \label{eq:06}
\max_{0\le t_n\le T}
\norm[L_2(\Omega, H)]{u^n-u(t_n)}
\le C\,\tau^{\frac\beta2}, \quad \text{if $0\le\beta<1$}.
\end{equation}

Finally, in Sect.\ \ref{sec:nonlinear}, we analyze the perturbation of the
nonlinear term and obtain that
\begin{equation} \label{eq:07}
 \max_{0\le t_n\le T}
 \norm[L_2(\Omega, H)]{\bar{v}^n-v^n}
 \le C  \max_{0\le t_n\le T}\norm[L_2(\Omega, H)]{w_J^n-w^n}.
\end{equation}

Therefore, our main result is the following.

\begin{theorem}\label{thm:main}
  Assume \eqref{eq:HSassumption} for some $\beta\ge0$.  Let
  $\{w_J^n\}_{n=0}^{N}\subset S_J$ be computed by a multiresolution
  Galerkin method of order $m$ and $\{v_\varepsilon^n\}_{n=0}^{N}$ by
  an adaptive wavelet method with tolerances $\varepsilon_n$.  Then
  for $0\le\beta < 1$, the total error in \eqref{output} is bounded by
\begin{align*}
\max_{0\le t_n\le T}
\norm[L_2(\Omega, H)]{u_\varepsilon^n-u(t_n)}
\le C\,\tau^{\frac\beta2} + C\, 2^{-J\min(\beta,m)} + C\sum_{n=1}^{N}\varepsilon_n,
\end{align*}
for $\tau L_f<\frac12$, where $C=C(\| u_{0} \|_{L_{2}( \Omega,
  \dot{H}^{\beta})}, \| A ^{\frac{\beta -1}{2}}
Q^{\frac12}\|_{\HS},T)$. If $\beta\ge1$, then first term is replaced by
$C_{\delta}\tau^{\frac12-\delta}$, for any $\delta>0$.
\end{theorem}

The literature on numerics for nonlinear  stochastic parabolic
problems is now rather large.  We mention, for example,
\cite{P2001} on pure time-discretization and \cite{kruse,yan} on
complete discretization based on the method of lines, where the
spatial discretization is first performed by finite elements and the
resulting finite-dimensional evolution problem is then discretized.
Wavelets have been used in \cite{Manuscript} where
the spatial approximation (without adaptivity) of
stochastic convolutions were considered.

Our present paper is a first attempt towards spatial adaptivity by
using Rothe's method together with known adaptive wavelet methods for
solving the resulting nonlinear elliptic problems.

The spatial Besov regularity of solutions of stochastic PDEs is
investigated in \cite{Dahlkeetal1,Dahlkeetal}.  The comparison of the
Sobolev and Besov regularity is indicative of whether adaptivity is
advantageous. For problems on domains with smooth or convex polygonal
boundary with boundary adapted additive noise (that is,
\eqref{eq:HSassumption} holds for $\beta$ high enough), where the
solution can be split as $u=v+w$, we expect that the adaptivity is not
needed for the stochastic convolution $w$, which then has sufficient Sobolev
regularity. We therefore apply it only to the random evolution problem
\eqref{eq:04}. Once the domain is not convex, or the boundary is not
regular, or the noise is not boundary adapted, adaptivity might pay
off also for the solution of the linear problem \eqref{Eq:3a}.

The recent paper \cite{Dahlkeetal2} is a first attempt for a more
complete error analysis of Rothe's method for both deterministic and
stochastic evolution problems.  The overlap with our present work is not
not too large, since we take advantage of special features of
equations with additive noise.



\section{Wavelet approximation}

In this section, we collect the notation and the main properties of
wavelets that will be needed in the sequel. We refer to
\cite{Cohen:2000,Dahmen:1997,Urban:2009} for more details on wavelet
methods for PDEs. For the space discretization, let
\begin{align*}
	\Psi = \{ \psi_\lambda: \lambda\in\cJ^\Psi\},\quad
	\tilde\Psi = \{ \tilde\psi_\lambda: \lambda\in\cJ^\Psi\}
\end{align*}
be a biorthogonal basis of $H$, i.e., in particular $(\psi_\lambda,\tilde\psi_\mu)_H=\delta_{\delta,\mu}$. Here, $\lambda$ typically is an index vector $\lambda=(j,k)$ containing both the information on the level $j=|\lambda|$ and the location in space $k$ (e.g., the center of the support of $\psi_\lambda$). Note that $\Psi$ also contains the scaling functions on the coarsest level that are not  wavelets. We will refer to $|\lambda|=0$ as the level of the \emph{scaling functions}.

In addition, we assume that $\psi_\lambda\in V$, which is an assumption on the regularity (and boundary conditions) of the primal wavelets. To be precise, we pose the following assumptions on the wavelet bases:
\begin{enumerate}
    \item Regularity: $\psi_\lambda\in H^t(\cD)$, $\lambda\in\cJ^\Psi$ for all $0\le t < s_\Psi$;
    \item Vanishing moments:   $( (\cdot)^r, \psi_\lambda)_{0;\cD} = 0$,  $0\le r< m_\Psi$, $|\lambda|>0$.
    \item Locality: $\operatorname{diam}(\Supp \psi_\Lambda)\sim 2^{-|\lambda|}$.
\end{enumerate}
We assume the same properties for the dual wavelet basis with $s_\Psi$ and $m_\Psi$ replaced by $\tilde s_\Psi$ and $\tilde m_\Psi$. Note that the dual wavelet $\tilde\psi_\lambda$ does not need to be in $V$, typically one expects $\tilde\psi_\lambda\in V'$.

We will consider (often finite-dimensional) subspaces generated by (adaptively generated finite) sets of indices $\Lambda\subset\cJ^\Psi$ and
$$
\Psi_\Lambda := \{ \psi_\lambda:\lambda\in\Lambda\},\quad
S_\Lambda := \operatorname{clos} \operatorname{span}(\Psi_\Lambda),
$$
where the closure is of course not needed if $\Lambda$ is a finite set. If $\Lambda=\Lambda_J:=\{ \lambda\in\cJ^\Psi:\, |\lambda|\le J-1\}$, then $S_J :=S_{\Lambda_J}$ contains all wavelets up to level $J-1$ so that $S_J$ coincides with the multiresolution space generated by all scaling functions on level $J$, i.e.,
\begin{align}
  \label{eq:mrs}
	S_J=\mathrm{span}\,\Phi_J, \quad \Phi_J=\{ \varphi_{J,k}: k\in\cI_J\},
\end{align}
where $\cI_J$ is an appropriate index set.
\subsection{Adaptive wavelet methods for nonlinear variational problems}
\label{Sec:2.2b}
In this section, we quote from \cite{CDD-Nonlinear} the main facts on adaptive wavelet methods for solving stationary nonlinear variational problems. Note, that all what is said in this section is taken from \cite{CDD-Nonlinear}. However, we abandon further reference for easier reading.

Let $F\colon V\to V'$ be a nonlinear map. We consider the problem of finding $u\in V$ such that
\begin{equation}\label{Eq:Non1}
	\langle v, R(u)\rangle := \langle v, F(u)-g\rangle = 0,\quad v\in V,
\end{equation}
where $g\in V'$ is given. As an example, let $F$ be given as $\langle
v,F(u)\rangle := a(v,u) + \langle v, f(u)\rangle$ which covers
\eqref{Eq:5}. The main idea is to consider an equivalent formulation
of \eqref{Eq:Non1} in terms of the wavelet coefficients $\bu$ of the
unknown solution $u=\bu^T\Psi$. Setting
$$
	\bR (\bv) := ( \langle\psi_\lambda, R(v)\rangle )_{\lambda\in\cJ^\Psi},\quad v=\bv^T\Psi,
$$
the equivalent formulation amounts to finding $\bu\in\ell_2(\cJ^\Psi)$ such that
\begin{equation}\label{Eq:Non2}
	\bR(\bv) = \mathbf{0}.
\end{equation}
The next ingredient is a basic iteration in the (infinite-dimensional) space $\ell_2(\cJ^\Psi)$ and  replacing the infinite operator applications in an adaptive way by finite approximations in order to obtain a computable version. Starting by some finite $\bu^{(0)}$, the iteration reads
\begin{equation}\label{Eq:Non3}
	\bu^{(n+1)} = \bu^{(n)} - \Delta\bu^{(n)},
	\quad
	\Delta\bu^{(n)} := \bB^{(n)} \bR(\bu^{(n)})
\end{equation}
where the operator $\bB^{(n)} $ is to be chosen and determines the nonlinear solution method (such as Richardson or Newton). The sequence $\Delta\bu^{(n)}= \bB^{(n)} \bR(\bu^{(n)})$ (possibly infinite even for finite input $\bu^{(n)}$) is then replaced by some finite sequence $\bw_\eta^{(n)}:= {\textbf{RES}} [\eta_n, \bB^{(n)}, \bR, \bu^{(n)}]$ such that
\begin{align*}
	\| \Delta\bu^{(n)} - \bw_\eta^{(n)}\| \le \eta_n.
\end{align*}

Replacing $\Delta\bu^{(n)}$  by $ \bw_\eta^{(n)}$ in \eqref{Eq:Non3} and choosing the sequence of tolerances $(\eta_n)_{n\in\N_0}$ appropriately results in a convergent algorithm such that any tolerance $\varepsilon$ is reached after finitely many steps. We set $\bar\bu(\varepsilon):=\textbf{SOLVE}[\varepsilon,\bR,\bB^{(n)},\bu^{(0)}]$ such that we get $\| \bu-\bar\bu(\varepsilon)||\le\varepsilon$.

In terms of optimality, there are several issues to be considered:\\[-1.0\baselineskip]
\begin{itemize}
\item How many iterations $n(\varepsilon)$ are required in order to achieve $\varepsilon$-accuracy?
\item How many ``active'' coefficients are needed to represent the numerical approximation and how does that compare with a ``best'' approximation?
\item How many operations (arithmetic, storage) and how much storage is required?
\end{itemize}

In order to quantify that, one considers so-called \emph{approximation classes}
$$
	\cA^s := \{ \bv\in\ell_2(\cJ^\Psi):\, \sigma_N(\bv)\lesssim N^{-s}\}
$$
of all those sequences whose \emph{error of best $N$-term approximation}
$$
	\sigma_N(\bv) := \min \{ \| \bv-\bw\|_{\ell_2}: \#\,\Supp\bw\le N\}
$$
decays at a certain rate ($\Supp\bv := \{ \lambda\in\cJ^\Psi:\, v_\lambda\not=0\}$, $\bv=(v_\lambda)_{\lambda\in\cJ^\Psi}$).

Let us first consider the case where $F=A$ is a linear elliptic partial differential operator, i.e., $Au= g\in V'$, where $A\colon  V\to V'$, $g\in V'$ is given and $u\in V$ is to be determined. For the discretization we use a wavelet basis $\Psi$ in $H$ where rescaled versions admit Riesz bases in $V$ and $V'$, respectively. Then, the operator equation can  equivalently be written as
\begin{align*}
    \bA\bu= \bg\in \ell_2 (\cJ^\Psi ),
\end{align*}
where $\bA:= \bD^{-1} a (\Psi, \Psi)\bD^{-1}$, $\bg := \bD^{-1} (g, \Psi)$ and $\bu: = \bD (u_\lambda)_{\lambda\in\cJ^\Psi}$, with $u_\lambda$ being the wavelet coefficients of the unknown function $u\in V$, $\Vert u\Vert_V\sim\Vert \bu\Vert_{\ell_2(\cJ^\Psi )}$. Wavelet preconditioning results in the fact that $\kappa_2(\bA) < \infty$, \cite{CDD}.

The (biinfinite) matrix $\bA$ is said to be \emph{$s^*$-compressible}, $\bA\in\cC_{s^*}$, if for any $0<s<s^*$ and every $j\in\N$ there exists a matrix $\bA_j$ with the following properties: For some summable sequence $(\alpha_j)_{j\in\N}$, the matrix $\bA_j$ is obtained by replacing all but the order of $\alpha_j 2^j$ entries per row and column in $\bA$ by zero and satisfies
$$
	\|\bA - \bA_j\| \le C \alpha_j 2^{-js},\quad j\in\N.
$$
Wavelet representations of differential (and certain integral) operators fall into this category. Typically, $s^*$ depends on the regularity and the order of vanishing moments of the wavelets. Then, one can construct a linear counterpart $\textbf{RES}_{\text{lin}}$ of $\textbf{RES}$ such that $\bw_\eta:=\textbf{RES}_{\text{lin}}[\eta,\bA,\bg,\bv]$ for finite input $\bv$ satisfies
\begin{subequations}
	\begin{align}
		\| \bw_\eta - (\bA\bv-\bg)\|_{\ell_2} &\le \eta, \label{Nonlin1a}\\
		\|\bw_\eta\|_{\cA^s} &\lesssim ( \|\bv\|_{\cA^s} + \| \bu\|_{\cA^s}), \label{Nonlin1b} \\
		\#\Supp \bw_\eta &\lesssim \eta^{-1/s}  ( \|\bv\|_{\cA^s}^{1/s} + \| \bu\|_{\cA^s}^{1/s}), \label{Nonlin1c}
	\end{align}
\end{subequations}
where the constants in \eqref{Nonlin1b}, \eqref{Nonlin1c} depend only on $s$. Here, we have used the quasi-norm
$$
	\| \bv\|_{\cA^s} := \sup_{N\in\N} N^s \sigma_N(\bv).
$$
This is the main ingredient for proving optimality of the scheme in the following sense.
\begin{theorem}[\cite{CDD,CDD-Nonlinear}]
	If $\bA\in\cC_{s^*}$ and if the exact solution $\bu$ of $\bA\bu=\bg$ satifies $\bu\in \cA^s$, $s<s^*$, then $\bar\bu(\varepsilon) = \bf{SOLVE}_{\rm{lin}}[\varepsilon]$ satisfies
\begin{subequations}
	\begin{align}
		\|\bu-\bar\bu(\varepsilon)\| &\le\varepsilon, \\
		\# \Supp \bar\bu(\varepsilon) &\lesssim \varepsilon^{-1/s},\\
		\mbox{computational complexity} \sim \# \Supp \bar\bu(\varepsilon).
	\end{align}
\end{subequations}
\end{theorem}

It turns out that most of what is said before also holds for the nonlinear case except that the analysis of the approximate evaluation of nonlinear expressions $\bR(\bv)$ poses a constraint on the structure of the active coefficients, namely that it has \emph{tree structure}. In order to define this, one uses the notation $\mu\prec\lambda$, $\lambda,\mu\in\cJ^\Psi$ to express that $\mu$ is a \textit{descendant} of $\lambda$. We explain this in the univariate case with $\psi_\lambda=\psi_{j,k} = 2^{j/2}\psi(2^j\cdot -k)$. Then, the \textit{children} of $\lambda=(j,k)$ are, as one would also intuitively define, $\mu=(j+1,2k)$ and $\nu=(j+1,2k+1)$. The descendants of $\lambda$ are its children, the children of its children and so on. In higher dimensions and even on more complex domains this can also be defined -- with some more technical effort, however.

Then, a set $\cT\subset\cJ^\Psi$ is called a \emph{tree} if $\lambda\in\cT$ implies $\mu\in\cT$ for all $\mu\in\cJ^\Psi$ with $\lambda\prec\mu$. Given this, the error of the \emph{best $N$-term tree approximation} is given as
\begin{align*}
	\sigma_N^{\rm{tree}}(\bv) := \min \{ \| \bv-\bw\|_{\ell_2}:
				\cT:= \#\,\Supp\bw\ \mbox{is a tree and } \#\cT\le N\}
\end{align*}
and define the \emph{tree approximation space} as
\begin{align*}
	\cA^s_{\rm{tree}} := \{ \bv\in\ell_2(\cJ^\Psi):\, \sigma_N^{\rm{tree}}(\bv)\lesssim N^{-s}\}
\end{align*}
which is a quasi-normed space under the quasi-norm
\begin{align*}
		\| \bv\|_{\cA^s_{\rm{tree}}} := \sup_{N\in\N} N^s \sigma_N^{\rm{tree}}(\bv).
\end{align*}
\begin{remark}\label{rem:bes}
  For the case $V=H^t$ (or, a closed subspace of $H^t$) it is known
  that the solution being in some Besov space $u\in B^{t+ds}_q(L_q)$,
  $q=(s+\frac12)^{-1}$, implies that $\bu\in\cA^r_{\rm{tree}}$, for
  $r<s$, see \cite[Remark 2.3]{CDDsparse}.
\end{remark}
The extension of the $s^*$-compressibility $\cC_{s^*}$ is the \emph{$s^*$-sparsity} of the scheme $\textbf{RES}$ which is defined by the following property: \sl
If the exact solution $\bu$ of \eqref{Eq:Non2} is in $\cA_{\rm{tree}}^s$ for some $s<s^*$, then $\bw_\eta := \bf{RES}[\eta,\bB,\bR,\bv]$ for finite $\bv$ satisfies
	\begin{align*}
		\|\bw_\eta\|_{\cA^s_{\rm{tree}}} &\le C ( \|\bv\|_{\cA^s_{\rm{tree}}} + \| \bu\|_{\cA^s_{\rm{tree}}}+1),\\
		\#\Supp \bw_\eta &\le C  \eta^{-1/s}  ( \|\bv\|_{\cA^s_{\rm{tree}}}^{1/s} + \| \bu\|_{\cA^s_{\rm{tree}}}^{1/s}+1), \\
		\mbox{comp.\ complexity} &\sim  C(  \eta^{-1/s}  ( \|\bv\|_{\cA^s_{\rm{tree}}}^{1/s} + \| \bu\|_{\cA^s_{\rm{tree}}}^{1/s}+1) +   \#\cT(\Supp \bv) ),
	\end{align*}
where $C$ depends only on $s$ when $s\to s^*$ and $\cT(\Supp\bv)$ denotes the smallest tree containing $\Supp \bv$.
\rm Now, we are ready to collect the main result.

\begin{theorem}[{\cite[Theorem 6.1]{CDD-Nonlinear}}]\label{Thm:CDD}
	If $\bf{RES}$ is $s^*$-sparse, $s^*>0$ and if $\bu\in\cA^s_{\rm{tree}}$ for some $s<s^*$, then the approximations $\bar\bu(\varepsilon)$ satisfy $\|\bu-\bu(\varepsilon)\|\le\varepsilon$ with
	$$
	\#\Supp \bar\bu(\varepsilon) \le C\, \varepsilon^{-1/s} \|\bu\|_{\cA^s_{\rm{tree}}}^{1/s},\quad
	\|\bar\bu(\varepsilon)\|_{\cA^s_{\rm{tree}}}\le C \|\bu\|_{\cA^s_{\rm{tree}}},
	$$
	where $C$ depends only on $s$ when $s\to s^*$. The number of operations is bounded by $C \varepsilon^{-1/s} \|\bu\|_{\cA^s_{\rm{tree}}}^{1/s}$.
\end{theorem}
We remark that since the wavelet transform is of linear complexity
the overall number of operations needed is the one mentioned in Theorem \ref{Thm:CDD}.

Next we show that the wavelet coefficients
$\mathbf{\bar{v}}^n$ of the solution of \eqref{Eq:5} belong to a certain approximation
class $\cA^s_{\rm{tree}}$ and hence, in view of Theorem \ref{Thm:CDD}, we obtain an estimate on the support
of $\bar\bv^n_\varepsilon$ and the number of operations required to
compute it.
\begin{theorem}\label{thm:complexity}
  The wavelet coefficients $\mathbf{\bar{v}}^n$ of the solution of
  \eqref{Eq:5} belong to $\cA^s_{\rm{tree}}$ for all
  $s<\frac{1}{2d-2}$, where $d\ge 2$ is the spatial dimension of
  $\mathcal{D}$.
\end{theorem}
\begin{proof}
  It follows from \cite[Lemma 5.15]{Dahlkeetal2} that $r(\tau A)\in
  \mathcal{L}(L_2(\mathcal{D}), B^{r}_q(L_q))$ for
  $r=\frac{3d-2+4\varepsilon}{2d-2+4\varepsilon}$, where
  $1/q=(r-1)/d+1/2$ and $\varepsilon>0$. Thus, the statement follows
  from Remark \ref{rem:bes} noting that $t=1$ and hence $r=1+ds$.
\end{proof}
We end this section by showing \eqref{Eq:5c}; that is, the overall
error after $n$ steps, when in every step \eqref{Eq:5} is solved
approximately up to an error tolerance $\varepsilon_n$ using the
adaptive wavelet algorithm described above.  Define
\begin{align*}
  E_j^n=E_n\circ \dots \circ E_{j+1}, \ E_n^n=I; \quad 0\le j<n\le N,
\end{align*}
and similarly $\tilde{E}_j^n$.  Then we have
\begin{align*}
  v_\varepsilon^n-\bar{v}^n
&
=\tilde{E}_0^n(u_0)-{E}_0^n(u_0)
\\ &
= \sum_{j=0}^{n-1}\big(
E_{j+1}^{n}( \tilde{E}_{0}^{j+1}(u_0) )
  - {E}_{j}^{n}( \tilde{E}_{0}^{j}(u_0) )\big)
\\ &
= \sum_{j=0}^{n-1}\big(
E_{j+1}^{n}(\tilde{E}_{j}^{j+1}(\tilde{E}_{0}^{j}(u_0)))
  -{E}_{j+1}^{n}({E}_{j}^{j+1}(\tilde{E}_{0}^{j}(u_0)))\big)
\\ &
= \sum_{j=0}^{n-1}\big(
E_{j+1}^{n}(\tilde{E}_{j+1}(v_\varepsilon^j))
  -{E}_{j+1}^{n}({E}_{j+1}(v_\varepsilon^j))\big) .
\end{align*}
A simple argument shows that the Lipschitz constant of $E_n$ is
bounded by $(1-\tau L_f)^{-1}\le \ee^{c\tau L_f}$ for some $c>0$, if $\tau L_f\le
\frac12$, cf.\ the proof of Lemma \ref{lem:s1}.  Hence $E_{j+1}^{n}$
has a Lipschitz constant bounded by $\ee^{c(t_n-t_{j+1})}\le \ee^{ct_N}$.
Thus, using \eqref{Eq:5b}, we obtain
\begin{align*}
    \norm{v_\varepsilon^n-\bar{v}^n}
 \le
 \sum_{j=0}^{n-1}\ee^{c(t_n-t_{j+1})}
\norm{\tilde{E}_{j+1}(v_\varepsilon^j)-{E}_{j+1}(v_\varepsilon^j) }
 \le
 \sum_{j=1}^{n}\ee^{c(t_n-t_{j})} \varepsilon_j
 \le \ee^{ct_N}  \sum_{j=1}^{n} \varepsilon_j.
\end{align*}
After taking a mean square we obtain \eqref{Eq:5c}.

\section{Error analysis for the stochastic convolution}
\label{sec:multires}

%

Let $S_J=S_{\Lambda_J}$ be a multiresolution space \eqref{eq:mrs}.
The multiresolution Galerkin approximation of the equation $Au=f$ in
$V'$ is to find $u_J\in S_J$ such that
\begin{align} \label{2.3.4}
  a(u_J,v_J)=(f,v_J)\quad \forall v\in S_J.
\end{align}
Define the orthogonal projector
$P_J \colon H\to S_J$ by
\begin{align}
\label{eq:charactPII}
(P_J v, w_J)= (v,w_J),\quad v\in H,\ w_J \in S_J .
\end{align}
Note that $P_J $ can be extended to $V'$ by \eqref{eq:charactPII}
since $S_J \subset V$.  Next, we define the operator $A_J \colon S_J
\to S_J$ by
\begin{align*}
a(A_J  v_J, w_J) = a(v_J, w_J), \quad u_J, v_J \in S_J.
\end{align*}
Then \eqref{2.3.4} reads
$A_J  u_J = P_J f\ \mbox{ in }\ S_J.$ Alternatively we may write $u_J=R_Ju$, where $R_J \colon V\to S_J$ is
the \emph{Ritz projector}, defined by
\begin{align*}
	a(R_J v, w_J) = a(v,w_J), \quad v\in V,\ w_J\in S_J.
\end{align*}

The multiresolution space is of order $m$ if
\begin{align}\label{eq:m}
 \inf_{w_J\in S_J}\norm{v-w_J}\lesssim 2^{-mJ}\norm[m;\cD]{v}, \quad
v\in H^m(\cD)\cap V.
\end{align}
Standard arguments then show, using elliptic regularity thanks to our assumptions on $\mathcal D$, that $\norm{u_J-u}\lesssim
2^{-mJ}\norm[m;\cD]{u}$, or in other words
\begin{align}
  \label{ritzerror}
 \norm{v-R_Jv}\lesssim 2^{-mJ}\norm[m;\cD]{v}, \quad v\in H^m(\cD)\cap V.
\end{align}

The next lemma
 is of independent interest and we state it in a general form.

\begin{lemma} \label{lem:euler-lemma} Let $-A$ and $-B$ generate
  strongly continuous semigroups $\ee^{-tA}$ and $\ee^{-tB}$ on a
  Banach space $X$ and let $r(s)=(1+s)^{-1}$. Then, for all $x,y\in
  X$, $N\in \mathbb{N}$, $\tau >0$,
\begin{align}\label{i}
\tau\sum_{n=1}^{N}\|r^n(\tau B)y-r^n(\tau A)x\|^p
&\le \int_0^{\infty}\|\ee^{-t B}y-\ee^{-t A}x\|^p\,\dd t, \quad 1\le p<\infty,
\\ \label{ii}
\|r^n(\tau B)y-r^n(\tau A)x\|
&\le \sup_{t\ge 0}\|\ee^{-tB}y-\ee^{-tA}x\|.
\end{align}
\end{lemma}

\begin{proof}
By the Hille-Phillips functional calculus, we have
\begin{equation}\label{iso}
r^n(\tau B)y-r^n(\tau A)x
=\int_{0}^{\infty}(\ee^{-t\tau B}y-\ee^{-t\tau A}x)f_n(t)\,\dd t,
\end{equation}
where $f_n$ denotes the $n$th convolution power of
$f(t)=\ee^{-t}$. Since $\|f_n\|_{L_1(\mathbb{R_+})}=1$ inequality
\eqref{ii} follows immediately by H\"older's inequality. To see
\eqref{i} we note that $f_n$ is a probability density and hence by
Jensen's inequality and \eqref{iso},
\begin{align*}
\tau \sum_{n=1}^{N}\|r^n(\tau B)y-r^n(\tau A)x\|^p
&=\tau \sum_{n=1}^{N}
\Big\|\int_{0}^{\infty}(\ee^{-t\tau B}y-\ee^{-t\tau A}x)f_n(t)\,\dd t \Big\|^p \\
& \le \tau \sum_{n=1}^{N}
\int_{0}^{\infty}\|\ee^{-t\tau B}y-\ee^{-t\tau A}x\|^p f_n(t) \,\dd t \\
&
= \int_{0}^{\infty}\|\ee^{-tB}y-\ee^{-tA}x\|^p\,\dd t\,
  \sup_{t> 0}\sum_{n=1}^\infty f_n(t).
\end{align*}
Finally, by monotone convergence, the Laplace transform of
$\sum_{n=1}^\infty f_n$ is given by
$$
{\Big(\sum_{n=1}^\infty f_n\Big)}\hat{\phantom{\Big|}}(\lambda)
=\sum_{n=1}^\infty \hat{f_n}(\lambda)
=\sum_{n=1}^\infty\Big(\frac{1}{1+\lambda}\Big)^n
=\frac{1}{\lambda},\quad\lambda>0.
$$
Thus, $\sum_{n=1}^\infty f_n\equiv 1$ and the proof is complete.
\end{proof}

Next we derive an error estimate for the
multiresolution approximation of the semigroup $\ee^{-tA}$ and its
Euler approximation $r^n(\tau A)$.

\begin{lemma} \label{lem:s2} Let $S_J$ be a multiresolution space of
order $m$ and let $A$, $A_J$, and $P_J$ be as above.  Then, for $T\ge0$, $N\ge1$, $\tau$, we have
\begin{equation}\label{eq:l2l2}
\Big(\int_{0}^{T} \|\ee^{-tA_J}P_Jv-\ee^{-tA}v\|^2\,\dd t  \Big)^{\frac12}
\le C 2^{-J\beta} \| v \|_{\beta -1},\quad 0\le \beta\le m,
\end{equation}
and
\begin{align}\label{eq:rl2l2}
\Big(\tau\sum_{n=1}^{N}\|r^n(\tau A_J)P_Jv-r^n(\tau A)v\|^2\Big)^{\frac12}
\le C 2^{-J\beta} \| v \|_{\beta -1},\quad 0\le \beta\le m.
\end{align}
\end{lemma}

\begin{proof}
  Estimate \eqref{eq:l2l2} is known in the finite element context, see
  for example \cite[Theorem 2.5]{Tho}, and may be proved in a
  completely analogous fashion for using the approximation property
  \eqref{ritzerror} of the Ritz projection $R_J$, the parabolic
  smoothing \eqref{eq:anal3}, and interpolation.  Finally,
  \eqref{eq:rl2l2} follows from \eqref{eq:l2l2} by using Lemma
  \ref{lem:euler-lemma} with $x=v$, $y=P_Jv$, and $B=A_J$.  (Note that
  $C$ is independent of $T$.)

\end{proof}

Now we are ready to consider the multiresolution approximation of
$w^n$ in \eqref{Eq:3a}.
\begin{theorem}\label{Thm:wn}
  Let $S_J$ be a multiresolution space of order $m$ and $w$ and
  $w_J^n$ the solutions of \eqref{Eq:3a} and \eqref{eq:multires}. If
  $\|A^{\frac{\beta-1}{2}}Q^{\frac12}\|_{\HS}<\infty$ for some $0\le
  \beta \le m$, then
\begin{align*}
(\bbE [\| w_J^n - w^n\|^2 ])^{\frac12}
\le C 2^{-J\beta}
\|A^{\frac{\beta-1}{2}}Q^{\frac12}\|_{\HS}.
\end{align*}
\end{theorem}

\begin{proof}
Let $t_k=k\tau$, $k=0,...,n$. By \eqref{eq:multires}, \eqref{Eq:3a},
and induction,
$$
w_J^n-w^n=\sum_{k=1}^n\int_{t_{k-1}}^{t_k}
\big[r^{n-k+1}(\tau A_J)P_J-r^{n-k+1}(\tau A)\big]\,\dd W(s),
$$
whence, by It\^o's isometry, we get
\begin{align*}
\bbE [\| w_J^n - w^n\|^2]
&=\sum_{k=1}^n\int_{t_{k-1}}^{t_k}
\big\|\big[r^{n-k+1}(\tau A_J)P_J-r^{n-k+1}(\tau A)\big]
Q^{\frac12}\big\|^2_{\HS}\,\dd s\\
&=\sum_{k=1}^n \tau
\big\|\big[r^{k}(\tau A_J)P_J-r^{k}(\tau A)\big]
Q^{\frac12}\big\|^2_{\HS}.
\end{align*}
Let $\{e_l\}_{l=1}^\infty$ be an orthonormal basis of $H$. Then, using
Lemma \ref{lem:s2}, we obtain
\begin{align*}
\bbE [\| w_J^n - w^n\|^2]
&=\sum_{l=1}^\infty \sum_{k=1}^n\tau
\|[r^k(\tau A_J)P_J-r^k(\tau A)]Q^{\frac12}e_l\|^2 \\
&\le C\sum_{l=1}^\infty 2^{-2J\beta} \|Q^{\frac{1}{2}}e_l\|_{\beta-1}^2
= C2^{-2J\beta}\|A^{\frac{\beta-1}{2}}Q^{\frac{1}{2}}\|_{\HS}^2.
\end{align*}
\end{proof}

\section{Pure time discretization} \label{sec:time}

In the proofs below we will often make use of
the following well-known facts about the analytic semigroup
$\ee^{-tA}$, namely
\begin{align}\label{eq:anal1}
\|A^{\alpha}\ee^{-tA}\|&\le {C}{t^{-\alpha}},\quad\alpha\ge 0,\ t>0, \\
\label{eq:anal2}
\|(\ee^{-tA}-I)A^{-\alpha}\|&\le Ct^{\alpha},\quad 0\le \alpha\le 1,\ t\ge 0,
\end{align}
for some $C=C(\alpha)$, see, for example, \cite[Chapter II, Theorem
6.4]{Pazy}. Also, by a simple energy argument we may prove
\begin{equation}\label{eq:anal3}
\int_0^t\|A^{\frac12}\ee^{-sA}v\|^2\,\dd s\le \frac12\|v\|^2,\quad v\in H,\ t\ge 0.
\end{equation}
We quote the following existence, uniqueness and stability result from
\cite[Lemma 3.1]{KLL}. For the mild, and other solution concepts we
refer to \cite[Chapters 6 and 7]{DapZab}.

\begin{lemma}\label{stab}
If $\| A^{\frac{\beta -1}{2}} Q^{\frac12} \|_{\HS}<\infty$ for some $\beta\ge
0$, $u_0\in L_2(\Omega,H)$, and \eqref{Lip} holds, then there is a
unique mild solution $\{u(t)\}_{t\ge 0}$ of \eqref{eq:1} with
$\sup_{t\in[0,T]}\bbE\|u(t)\|^2\le K$, where
$K=K(u_0,T,L_f)$.
\end{lemma}

Concerning the temporal regularity of the stochastic convolution we
have the following theorem.

\begin{theorem}\label{Holdersconv}
  Let $\|A^{-\eta}Q^{\frac{1}{2} }\|_{\HS}<\infty$ for some $\eta \in [0,
  \frac{1}{2}]$. Then the stochastic convolution $w(t):=
  \int_0^t\ee^{-(t-\sigma)A}\,\dd W(\sigma)$ is mean square H\"older
  continuous on $[0,\infty)$ with H\"older constant $C=C(\eta)$ and
  H\"older exponent $\frac12-\eta$, i.e.,
\begin{equation*}
\left(\bbE \|w (t)-w (s)\|^2\right)^{\frac12}
\leq C|t-s|^{\frac12-\eta},\quad t,s\ge 0.
\end{equation*}
\end{theorem}
\begin{proof}
  For $\eta=\frac12$ the result follows from Lemma \ref{stab}. Let
  $\eta \in [0, \frac{1}{2})$ and, without loss of generality, let
  $s < t$.  By independence of the increments of $W$,
  \begin{equation*}
\begin{split}
\ens{w(t)-w(s)}
& =\Ens{\int_s^t \ee^{-(t-\sigma)A}\, \dd W(\sigma)}
\\ &\quad +\Ens{\int_0^s \ee^{-(t-\sigma)A}-\ee^{-(s-\sigma)A}\, \dd W(\sigma)}
= I_1+I_2.
\end{split}
\end{equation*}
From It\^o's isometry and \eqref{eq:anal1} it follows that
\begin{equation*}
\begin{split}
I_1&
=\Ens{\int_s^t A^{\eta}\ee^{-(t-\sigma)A}A^{-\eta}\,\dd W(\sigma)}
=\int_s^t\shs{ A^{\eta}\ee^{-(t-\sigma)A}A^{-\eta}Q^{\frac{1}{2} }}\,\dd\sigma\\
& \leq C\int_s^t(t-\sigma)^{-2\eta}\shs{A^{-\eta}Q^{\frac{1}{2} }}\, \dd\sigma
\leq \frac{C}{1-2\eta}(t-s)^{1-2\eta}\shs{A^{-\eta}Q^{\frac{1}{2} }}.
\end{split}
\end{equation*}
Finally, let $\{e_k\}_{k=1}^\infty$ be an orthonormal basis of $H$. Then,
by \eqref{eq:anal2} and \eqref{eq:anal3},
\begin{equation}\label{i2}
\begin{split}
I_2&=\int_0^s\shs{(\ee^{-(t-\sigma)A}-\ee^{-(s-\sigma)A})Q^{\frac{1}{2} }} \,\dd\sigma\\
&=\sum_{k=1}^\infty
\int_0^s\|(\ee^{-(t-s)A}-I)A^{-(\frac{1}{2} -\eta)}A^{\frac{1}{2} -\eta}
  \ee^{-(s-\sigma)A}Q^{\frac{1}{2} }e_k\|^2 \,\dd\sigma \\
 & \le C(t-s)^{1-2\eta} \sum_{k=1}^\infty
\int_0^s\|A^{\frac{1}{2} }\ee^{-(s-\sigma)A}A^{-\eta}Q^{\frac{1}{2} }e_k\|^2 \,\dd\sigma\\
 & \le C (t-s)^{1-2\eta} \sum_{k=1}^\infty
 \|A^{-\eta}Q^{\frac{1}{2} }e_k\|=C (t-s)^{1-2\eta} \|A^{-\eta}Q^{\frac{1}{2} }\|^2_{\HS}.
\end{split}
\end{equation}
\end{proof}

The next result shows that the time regularity of $w$ transfers to the solution of the
semilinear problem.

\begin{theorem}\label{Thm:hol}
  If $u_0\in L_2(\Omega,\dot{H}^{\beta})$ and
  $\|A^{\frac{\beta-1}{2}}Q^{\frac12}\|_{\HS}<\infty$ for some $0\le
  \beta < 1$, then there is
  $C=C(T,u_0,\beta)$ such that the mild solution $u$ of \eqref{eq:1} satisfies
  \begin{equation*}
  \left(\bbE\|u(t)-u(s)\|^2\right)^{\frac12}
  \le C|t-s|^{\frac{\beta}{2}},\quad t,s\in [0,T].
  \end{equation*}
\end{theorem}

\begin{proof}
Let $T>0$ and $0\le s< t\le T$. Then, by \eqref{eq:01},
\begin{align*}
u(t)-u(s)&=(\ee^{-tA}-\ee^{-sA})u_0
+\int_s^t\ee^{-(t-r)A}f(u(r))\,\dd r\\
& \quad +\int_0^s(\ee^{-(t-r)A}-\ee^{-(s-r)A})f(u(r))\,\dd r
+w(t)-w(s).
\end{align*}
In a standard way, for $0\le \beta\le 2$, we have
$
\bbE\|(\ee^{-tA}-\ee^{-sA})u_0\|^2\le C|t-s|^{\beta}\bbE\|u_0\|_{\beta}^2.
$
Using that $f$ is Lipschitz and hence $\|f(u)\|\le C(1+\|u\|)$, the
norm boundedness of the semigroup $\ee^{-tA}$, and Lemma \ref{stab},
we have that
$$
\bbE\Big\|\int_s^t\ee^{-(t-r)A}f(u(r))\,\dd r\Big\|^2
\le C|t-s|^2 \Big(1+\sup_{r\in [0,T]}\bbE\|u(r)\|^2\Big)\le C|t-s|^2.
$$
For $0\le \beta<1$, by Lemma \ref{stab}, \eqref{eq:anal1} and
\eqref{eq:anal2}, it follows that
\begin{align*}
&\bbE\Big\|\int_0^s(\ee^{-(t-r)A}-\ee^{-(s-r)A})f(u(r))\,\dd r\Big\|^2 \\
&\qquad \le s \bbE \int_0^s\|(\ee^{-(t-r)A}-\ee^{-(s-r)A})f(u(r))\|^2\,\dd r \\
&\qquad
\le C s\Big(1+\sup_{r\in [0,T]}\bbE\|u(r)\|^2\Big)
\int_0^s\|\ee^{-(t-r)A}-\ee^{-(s-r)A}\|^2\,\dd r\\
&\qquad \le
C s\int_0^s\|A^{\frac{\beta}{2}}\ee^{-(s-r)A}
(\ee^{-(t-s)A}-I)A^{-\frac{\beta}{2}}\|^2\,\dd r
\le C |t-s|^{\beta}s^{2-\beta}\le C |t-s|^{\beta}.
\end{align*}
Finally, by Theorem \ref{Holdersconv} with $\eta =-\frac{\beta-1}{2}$, we have
$
\bbE\|w(t)-w(s)\|^2\le C |t-s|^{\beta},
$
which finishes the proof.
\end{proof}
In order to analyze the order of the backward Euler time-stepping
\eqref{Eq:2}  we quote the following deterministic error
estimates, where $r(\tau A)=(I+\tau
A)^{-1}$.

\begin{lemma}
The following error estimates hold for $t_n=n\tau>0$.
\begin{align}\label{Eq:tmd0}
\|[\ee^{-n\tau A}-r^n(\tau A)]v\|&\le C\tau^{\frac\beta2}\|v\|_{\beta},
\quad 0\le \beta \le 2,\\
\label{Eq:tmd01}
\|[\ee^{-n\tau A}-r^n(\tau A)]v\|&\le C \tau t_n^{-1}\|v\|,\\
\label{Eq:tmd}
\sum_{k=1}^n\tau\left\|[r^k(\tau A)-\ee^{-k\tau A}]v\right\|^2
&\le C \tau^\beta \|v\|^2_{\beta-1},\quad 0\le \beta\le 2.
\end{align}
\end{lemma}

\begin{proof}
  Estimates \eqref{Eq:tmd0} and \eqref{Eq:tmd01} are shown in, for
  example, \cite[Chapter 7]{Tho}. Estimate \eqref{Eq:tmd} can be
  proved in a similar way as (2.17) in
  \cite[Lemma 2.8]{yan}.
\end{proof}

\begin{theorem}\label{Thm:tme}
  If $u_0\in L_2(\Omega,\dot{H}^{\beta})$ and
  $\|A^{\frac{\beta-1}{2}}Q^{\frac12}\|_{\HS}<\infty$ for some $0\le
  \beta < 1$, then there is $C=C(T, u_0,\beta)$ such that for
  $0<\tau<\frac{1}{2L_f}$, the solutions $u$ of \eqref{eq:01} and $u^n$
  of \eqref{Eq:2} satisfy
$$
(\bbE\|u(t_n)-u^n\|^2)^{\frac12}\le C \tau^{\beta/2},\quad t_n=n\tau\in [0,T].
$$
\end{theorem}

\begin{proof}
We have, with $e^n:=u(t_n)-u^n$,
\begin{align*}
e^n&=[\ee^{-t_nA}-r^n(\tau A)]u_0
+\sum_{k=1}^n
\int_{t_{k-1}}^{t_k} [\ee^{-(t_n-s)A}-r^{n-k+1}(\tau A)]\,\dd W(s)\\
&\quad +\sum_{k=1}^n
\int_{t_{k-1}}^{t_k}
\ee^{-(t_n-s)A}f(u(s))-r^{n-k+1}(\tau A)f(u_{k})\,\dd s
=e_1+e_2+e_3.
\end{align*}
The error $e_1$ is
easily bounded, using \eqref{Eq:tmd0}, as
\begin{equation*}
\bbE\|e_1\|^2\le C \tau^\beta \bbE\|u_0\|^2_{\beta},\quad 0\le \beta \le 2.
\end{equation*}
The contribution of $e_2$ is the linear stochastic error. First, we
decompose $e_2$ as
\begin{align*}
e_2&=\sum_{k=1}^n\int_{t_{k-1}}^{t_k}
[\ee^{-t_{n-k+1}A}-r^{n-k+1}(\tau A)]\,\dd W(s)\\
&\quad +\sum_{k=1}^n\int_{t_{k-1}}^{t_k}
[\ee^{-(t_n-s)A}-\ee^{-t_{n-k+1}A}]\,\dd W(s)=e_{21}+e_{22}.
\end{align*}
Let $\{f_l\}_{l=1}^\infty$ be an ONB of $H$. By It\^o's isometry, the
independence of the increments of $W$ and \eqref{Eq:tmd},
\begin{align*}
\bbE\|e_{21}\|^2&= \sum_{k=1}^n\tau \|[r^{k}(\tau A)-\ee^{-k\tau A}]Q^{\frac12}\|^2_{\HS}\le
=\sum_{l=1}^\infty\sum_{k=1}^n\tau \|[r^{k}(\tau A)-\ee^{-k\tau A}]Q^{\frac12}f_l\|^2\\
&\le C\sum_{l=1}^\infty\tau^{\beta}\|Q^{\frac12}f_l\|^2_{\beta-1}
=C\tau^\beta\|A^{\frac{\beta-1}{2}}Q^{\frac{1}{2}}\|_{\HS}^2,\quad 0\le \beta\le 2.
\end{align*}
The term $e_{22}$ can be bounded using a similar argument as in \eqref{i2} by
\begin{equation*}
\bbE\|e_{22}\|^2\le C\tau^{\beta}\|A^{\frac{\beta-1}{2}}Q^{\frac12}\|^2_{\HS},\quad 0\le \beta\le 2.
\end{equation*}
Next, we can further decompose $e_3$ as
\begin{align*}
e_3&=\sum_{k=1}^n\int_{t_{k-1}}^{t_k}r^{n-k+1}(\tau A)[f(u(t_{k}))-f(u_{k})]\,\dd s\\
&\quad +\sum_{k=1}^n\int_{t_{k-1}}^{t_k}[\ee^{-t_{n-k+1}A}-r^{n-k+1}(\tau A)]f(u(t_{k}))\,\dd s\\
&\quad +\sum_{k=1}^n\int_{t_{k-1}}^{t_k}\ee^{-t_{n-k+1}A}[f(u(s))-f(u(t_{k}))]\,\dd s\\
&\quad +\sum_{k=1}^n\int_{t_{k-1}}^{t_k}[\ee^{-(t_n-s)A}-\ee^{-t_{n-k+1}A}]f(u(s))\,\dd s
=e_{31}+e_{32}+e_{33}+e_{34}.
\end{align*}
By the stability of $r^n(\tau A)$ and the Lipschitz condition on $f$,
we have
\begin{align*}
\bbE\|e_{31}\|^2
\le 2L_f^2\tau^2 \bbE\|e^n\|^2
+2L_f^2\tau^2n\sum_{k=1}^{n-1}\bbE\|e^k\|^2
\le 2L_f^2\tau^2 \bbE\|e^n\|^2
+C\tau \sum_{k=1}^{n-1}\bbE\|e^k\|^2.
\end{align*}
By \eqref{Eq:tmd01} and Lemma \ref{stab}, with $\tau t_{n-k+1}^{-1}=(n-k+1)^{-1}=l^{-1}$,
\begin{align*}
\bbE\|e_{32}\|^2
&\le C\bbE\Big(\sum_{k=1}^n\tau \tau t_{n-k+1}^{-1}\|f(u(t_k))\|\Big)^2
\le C \tau^2 \sum_{l=1}^n\frac{1}{l^{2}}\sum_{k=1}^n\bbE\|f(u(t_k))\|^2\\
&\le C \tau^2 \sum_{k=1}^n(1+\bbE\|u(t_k)\|^2)\le C \tau t_n\le C\tau.
\end{align*}
Furthermore, by Theorem \ref{Thm:hol},
\begin{equation*}
\bbE\|e_{33}\|^2\le t_n \sum_{k=1}^n
\int_{t_{k-1}}^{t_k}\bbE\|f(u(s))-f(u(t_k))\|^2\,\dd s
\le Ct_n^2 \tau^\beta \le C \tau^\beta,~0\le \beta< 1.
\end{equation*}
To estimate $e_{34}$ we have, using again that $t_{n-k+1}=t_n-t_{k-1}$
and Lemma \ref{stab},
\begin{align*}
\bbE\|e_{34}\|^2
&= \bbE\Big(\sum_{k=1}^n\int_{t_{k-1}}^{t_k}
\|[A^{\frac{\beta}{2}}\ee^{-(t_n-s)A}(I-\ee^{-(s-t_{k-1})A})]
A^{-\frac{\beta}{2}}f(u(s))\|\,\dd s\Big)^2\\
&\le C t_n\sum_{k=1}^n
\int_{t_{k-1}}^{t_k}(t_n-s)^{-\beta}\tau^\beta
\bbE\|f(u(s))\|^2\,\dd s\le C\tau^\beta,\quad 0\le \beta< 1.
\end{align*}
Putting the pieces together, we have
\begin{align*}
\bbE\|e^n\|^2
\le C\tau^\beta
+2L_f^2\tau^2 \bbE\|e^n\|^2
+C\tau \sum_{k=1}^{n-1}\bbE\|e^k\|^2,\quad 0\le\beta< 1.
\end{align*}
Finally, if $\tau<\frac{1}{2L_f}$, then by the discrete Gronwall lemma,
\begin{equation*}
\bbE\|e^n\|^2\le C\tau^\beta \ee^{Ct_n}\le C\tau^\beta,~0\le\beta< 1,
\end{equation*}
and the theorem is established.
\end{proof}

\section{Error analysis for the nonlinear random problem}
\label{sec:nonlinear}
In this section we bound the term $\bbE \big[\|\bar{v}^n-v^n \|^2\big]$ in
\eqref{eq:07}. We use the global Lipschitz condition \eqref{Lip}.

\begin{lemma}  \label{lem:s1}
Assume that $\tau L_f\le \frac12$.  Then, with $C=2L_fTe^{2L_fT}$,
\begin{align*}
\max_{1\le n\le N}
\Big(\bbE \big[ \| \bar{v}^n-v^n \|^2 \big]\Big)^{\frac12}
\leq C
\max_{1\le n\le N}
\Big(\bbE \big[ \| w_J^n - w^n\|^2 \big]\Big)^{\frac12} .
\end{align*}
\end{lemma}

\begin{proof}
Let $e^n:= \bar{v}^n - v^n$.
Then, we have by \eqref{Eq:3b} and \eqref{Eq:5}
\begin{align*}
e^n + \tau Ae^n
= \tau \big(f(\bar v^n + w_J^n)-f(v^n + w^n)\big) + e^{n-1}.
\end{align*}
Since $e^0=0$, we get by induction
\begin{align*}
e^n
= \tau \sum_{j=1}^{n} (I + \tau A)^{-(n+1-j)}
\big(f(\bar v^j + w_J^j)-f(v^j + w^j)\big).
\end{align*}
In view of the global Lipschitz condition \eqref{Lip},
this results in the estimate
\begin{align*}
\| e^n\|
&\leq L_f\tau \sum_{j=1}^{n}
\|(I + \tau A)^{-(n+1-j)}\|\,\| \bar v^j + w_J^j -v^j - w^j \|
\\
&\leq L_f \tau \sum\limits_{j=1}^{n} \big(\| w_J^j - w^j\|+\| e^j\| \big),
\end{align*}
since $\| (I+\tau A)^{-1}\| \leq 1$.
Thus, we obtain
\begin{align*}
\| e^n\|
\leq (1-L_f\tau)^{-1}L_f \tau\Big(
  \sum_{j=1}^{n} \| w_J^j - w^j\|
+ \sum_{j=1}^{n-1} \| e^j\|
\Big).
\end{align*}
With $L_f\tau \le \frac{1}{2}$ we complete the proof by the standard
discrete Gronwall lemma.
\end{proof}



\end{document}